\documentclass{article}
\usepackage{amssymb,amsmath,graphicx,ucs,cite}
\usepackage[utf8x]{inputenc}
\usepackage{hyperref}

\newtheorem{theorem}{Theorem}
\newtheorem{definition}{Definition}
\newtheorem{proposition}{Proposition}
\newtheorem{lemma}{Lemma}

\newenvironment{proof}{\noindent \emph{Proof. }}{\hfill \hbox{\rlap{$\sqcap$}$\sqcup$}\\}

\title{No Weak Local Rules for the $4p$-Fold Tilings\thanks{This work was supported by the ANR project QuasiCool (ANR-12-JS02-011-01)}}

\author{
Nicolas Bédaride
\footnote{Aix Marseille Univ., CNRS, Centrale Marseille, I2M, UMR 7373, 13453 Marseille, France.}
\and
Thomas Fernique
\footnote{Univ. Paris 13, CNRS, Sorbonne Paris Cité, UMR 7030, 93430 Villetaneuse, France.}
}
\date{}

\begin{document}

\maketitle

\begin{abstract}
  On the one hand, Socolar showed in 1990 that the $n$-fold planar tilings admit weak local rules when $n$ is not divisible by $4$ (the $n=10$ case corresponds to the Penrose tilings and is known since 1974).
  On the other hand, Burkov showed in 1988 that the $8$-fold tilings do not admit weak local rules, and Le showed the same for the $12$-fold tilings (unpublished).
  We here show that this is actually the case for all the $4p$-fold tilings.
\end{abstract}

%%%%%%%%%%%%%%%%%%%%%%%%%%
\section{Introduction}

Quasicrystals are ordered but nevertheless non-periodic materials.
Their structure is commonly modeled by {\em tilings}, that are covering of the Euclidean plane or space by non-overlapping compact sets called {\em tiles}.
The interesting structure of numerous quasicrystals is actually only two-dimensional, with the third dimension corresponding to periodically stacked arrangement of atoms.
This explains why the tilings of the plane have retained no less attention than the tilings of the space -- and we do focus here on the former.
When the tiles are moreover rhombi, one speaks about {\em rhombus tilings}.
The rhombus tilings have the remarkable pro\-perty that they can be {\em lifted} in a higher dimensional space.
In particular, those whose lift stay at bounded distance from an affine plane are said to be {\em planar}: they have a long range order which make them especially suitable to model the structure of quasicrystals.\\
%Among them, {\em $n$-fold tilings} have in addition some physically interesting symmetry properties

As for any material, understanding a quasicrystal means not only understanding its structure but also its stability, that is, how finite-range energetic interactions make the atoms achieving such a structure.
In terms of tilings, this means understanding how constraints on the way neighbor tiles can fit together -- one speaks about {\em local rules} -- enforce the planarity of a tiling.
Local rules can be formally defined in several ways.
Here, we shall follow Levitov \cite{levitov}, who considered {\em undecorated} local rules, one of the simplest model.
For the planar rhombus tilings, Levitov also introduced {\em weak} and {\em strong} local rules, the formal definition of which shall be further recalled.
In this context, the goal is to find a characterization of the planar rhombus tilings which admit undecorated weak local rules.
This remains an open problem.
Let us however mention that such a characterization has been recently obtained when {\em decorated} local rules are allowed (see \cite{FS}).
In terms of symbolic dynamics, the tiling sets defined by undecorated or decorated local rules are respectively called {\em tiling spaces of finite type} or {\em sofic tiling spaces}  (see \cite{robinson}).\\

Among the several conditions on the planar rhombus tilings with (undecorated) weak or strong local rules that have been found (\cite{BF,subperiods,burkov,katz,KP,le92,le92b,le92c,le93,le95,levitov,socolar}), we are interested in thoses which deal with $n$-fold tilings.
In \cite{socolar}, Socolar proved that the $n$-fold tilings admit weak local rules as soon as $n$ is not a multiple of $4$.
This disproved the common belief that whenever a planar rhombus tiling admits weak local rules, then the plane its lift stays at bounded distance of can always be defined by {\em quadratic} irrationalities (irrationalities are cubic already for $n=7$). 
Socolar moreover explicitly derived simple local rules from what he called the {\em alternation condition}.
Without going into details, this condition states that each rhombus tile must ``alternate'' in a specific way with its mirror image with respect to one of its edges.
The problem with the $4p$-fold tilings is that they have square tiles which are equal to their own mirror image! 
Actually, Burkov proved in \cite{burkov} that the $8$-fold tilings, also known as the Ammann-Beenker tilings, do not admit weak local rules\footnote{Note that it admits {\em decorated} local rules, as proved by Robert Ammann himself, see \cite{AGS,socolar2}}.
To prove this, he provided a one-parameter fa\-mi\-ly of planar rhombus tilings which contains the $8$-fold tilings, and such that the closer the parameter is to the one of the $8$-fold tilings, the larger is the smallest pattern which allows to distinguish the tilings corresponding to each parameter.
We here extend this by providing, for each $p$, such a one-parameter family for the $4p$-fold tiling.
This yields our main result:

\begin{theorem}\label{th:main}
The $4p$-fold tilings do not admit weak local rules.
\end{theorem}

Let us briefly describe the two main tools that shall be used to prove this.
The first one is the notion of {\em window}, which is classic in the context of so-called {\em cut and project tilings}.
It is a convenient tool to study the patterns that appear in a planar rhombus tilings, and we shall especially rely on results obtained by Julien in \cite{julien}.
The se\-cond tool is the notion of {\em subperiod}, introduced by the authors in \cite{BF,subperiods} and which corresponds to the {\em second-intersection condition} earlier introduced by Levitov in \cite{levitov} and used, {\em e.g.}, by Le in \cite{le92c}.
Roughly speaking, a subperiod is a rational dependency between {\em some} of the entries of vectors which generate a (possibly irrational) plane.
This is the notion that led us to the one-paramater families of planar rhombus tilings that is used to show Theorem~\ref{th:main}.\\

The paper is organized as follows.
In Section~\ref{sec:settings}, we formally define the above mentioned notions: rhombus tilings and their lift in a higher dimensional space, planar tilings, $n$-fold tilings, weak local rules and subperiods.
We also review some basic properties of Grassmann coordinates.
In Section~\ref{sec:subperiods}, we define the one-parameter families of planar rhombus tilings that is used to show Theorem~\ref{th:main}.
In Section~\ref{sec:window}, we briefly recall known results on the window of a planar tiling and introduce the notion of {\em coincidence}.
We finally prove Theorem~\ref{th:main} in Section~\ref{sec:coincidences}.
%%%%%%%%%%%%%%%%%%%%%%%%%%
\section{Settings}
\label{sec:settings}

\paragraph{Rhombus tiling.}
Let $\vec{v}_1,\ldots,\vec{v}_n$ be $n\geq 3$ pairwise non-collinear unit vectors of the Euclidean plane.
They define the $\binom{n}{2}$ rhombus {\em prototiles}
$$
T_{ij}=\{\lambda\vec{v}_i+\mu\vec{v}_j~|~0\leq\lambda,\mu\leq 1\}.
$$
A {\em tile} is a translated prototile (tile rotation or reflection are forbidden).
A {\em rhombus tiling} is a covering of the Euclidean plane by interior-disjoint tiles satisfying the {\em edge-to-edge} condition: whenever the intersection of two tiles is not empty, it is either a vertex or an entire edge.

\paragraph{Lift.}
Let $\vec{e}_1,\ldots,\vec{e}_n$ be the canonical basis of $\mathbb{R}^n$.
A rhombus tiling is {\em lifted} in $\mathbb{R}^n$ as follows: an arbitrary vertex is first mapped onto the origin of $\mathbb{R}^n$, then each tile $T_{ij}$ is mapped onto the $2$-dimensional face of a unit hypercube of $\mathbb{Z}^n$ generated by $\vec{e}_i$ and $\vec{e}_j$, with two tiles adjacent along an edge $\vec{v}_i$ being mapped onto two faces adjacent along an edge $\vec{e}_i$.
This lifts the boundary of a tile -- and by induction the boundary of any patch of tiles -- onto a closed curve of $\mathbb{R}^n$ and hence ensures that the image of a tiling vertex does not depend on the path followed to get from the origin to this vertex.
The lift of a tiling is thus a ``stepped'' surface in $\mathbb{R}^n$ (unique up to the choice of the initial vertex).

\paragraph{Planar tiling.}
A rhombus tiling is said to be {\em planar} if there is a $t\geq 1$ and an affine plane $E\subset \mathbb{R}^n$ such that the tiling can be lifted into the tube $E+[0,t]^n$ (we need $t\geq 1$ to have complete tiles in the tube).
The smallest suitable $t$ is called the {\em thickness} of the tiling, and the corresponding $E$ is called the {\em slope} of the tiling.
Both are uniquely defined.
A planar rhombus tiling is thus an {\em approximation} of its slope: the less the thickness, the better the approximation.

\paragraph{$n$-fold tiling.}
For $n\geq 4$ even, the {\em $n$-fold tilings} are the thickness $1$ planar tilings whose slope is generated by the vectors whose $k$-th entry are respectively $\cos(2k\pi/n)$ and $\sin(2k\pi/n)$, for $0\leq k<n/2$.
The lift of a $n$-fold tiling thus lives in $\mathbb{R}^{n/2}$.
The name comes from the fact that they admit a local $n$-fold rotational symmetry: any finite pattern of such a tiling indeed also appears in its image under a rotation by $2\pi/n$.
Fig.~\ref{fig:n_fold_tilings} illustrates this.

\begin{figure}[hbtp]
\includegraphics[width=0.3\textwidth]{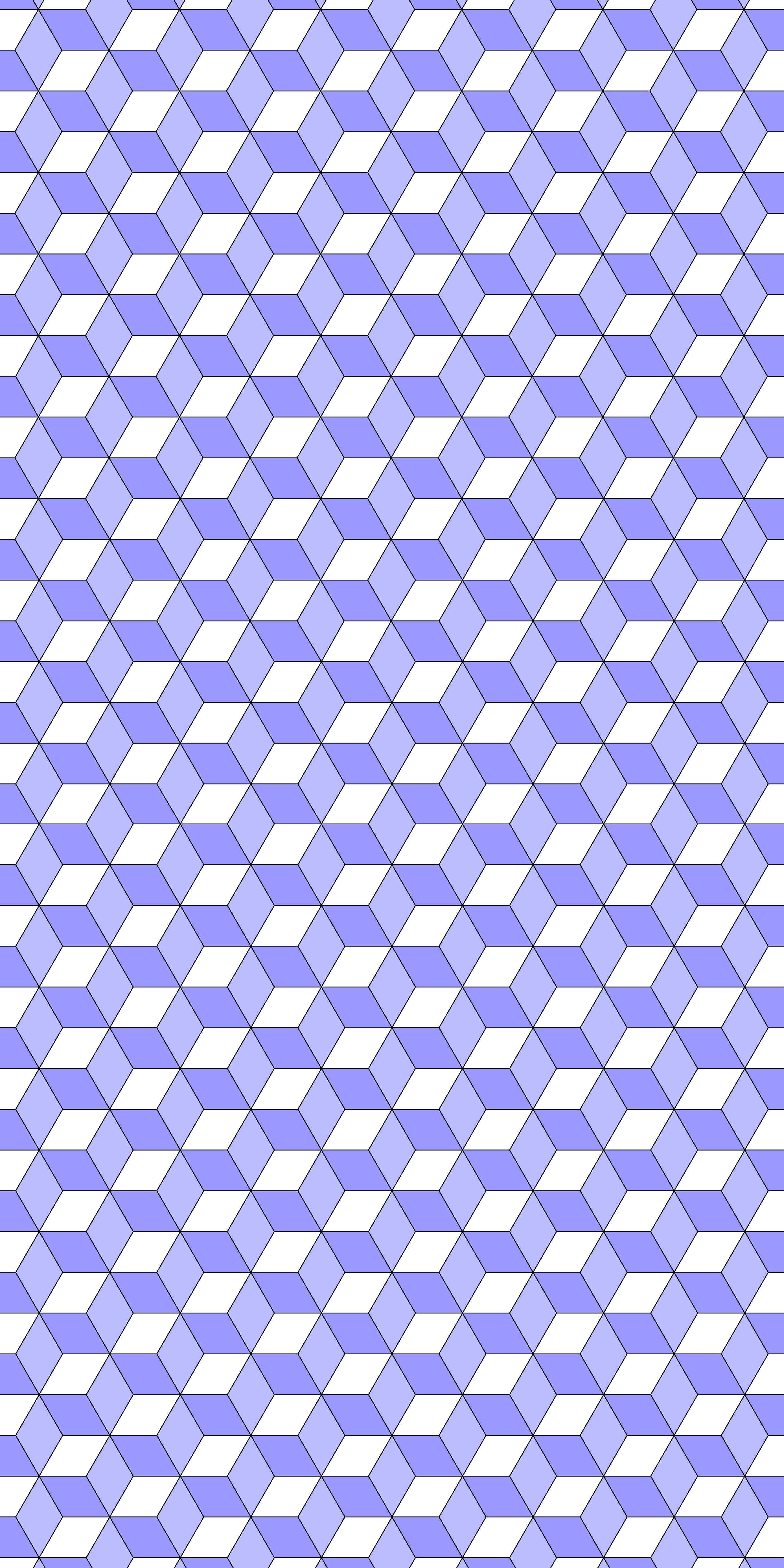}\hfill
\includegraphics[width=0.3\textwidth]{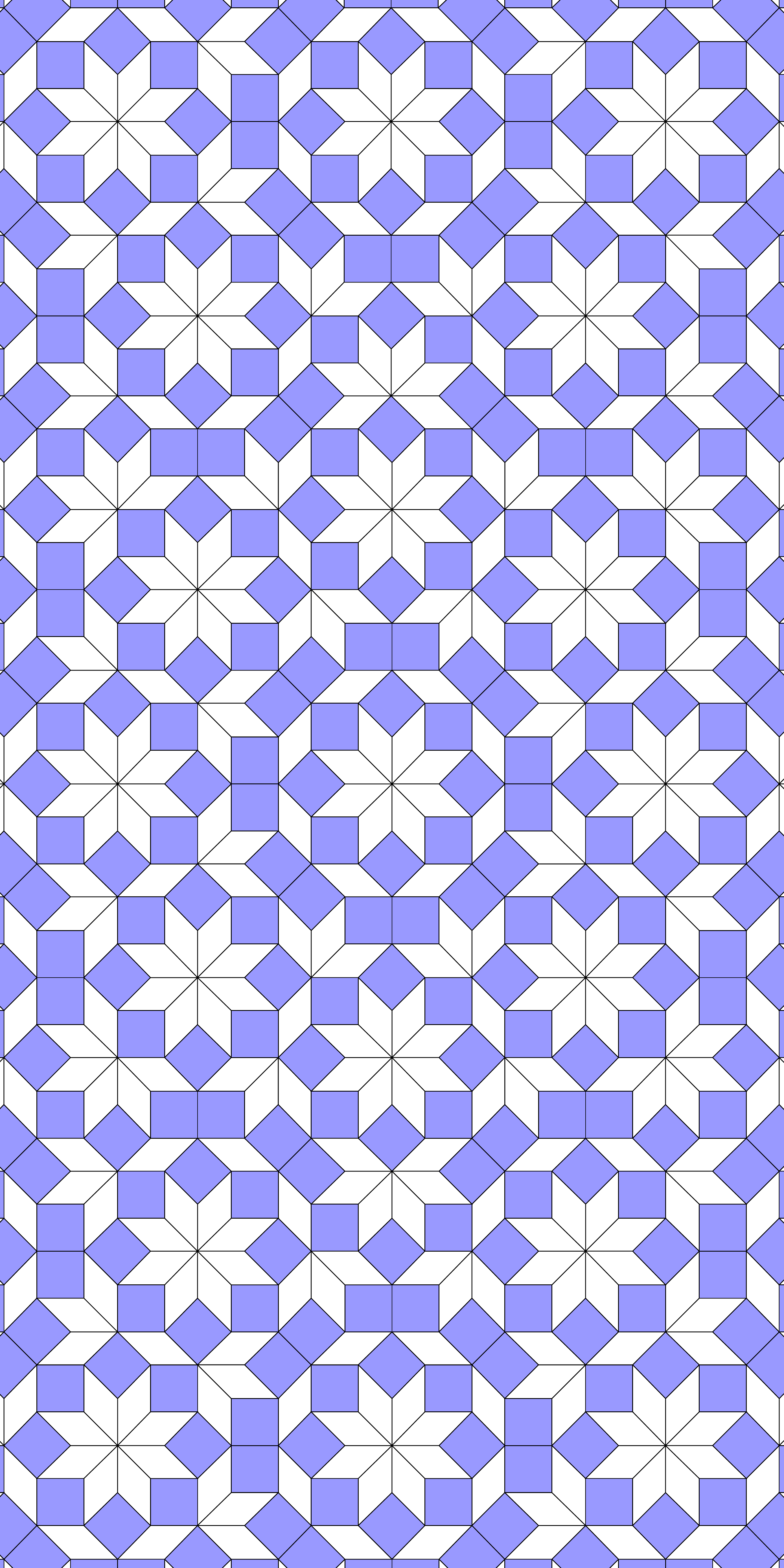}\hfill
\includegraphics[width=0.3\textwidth]{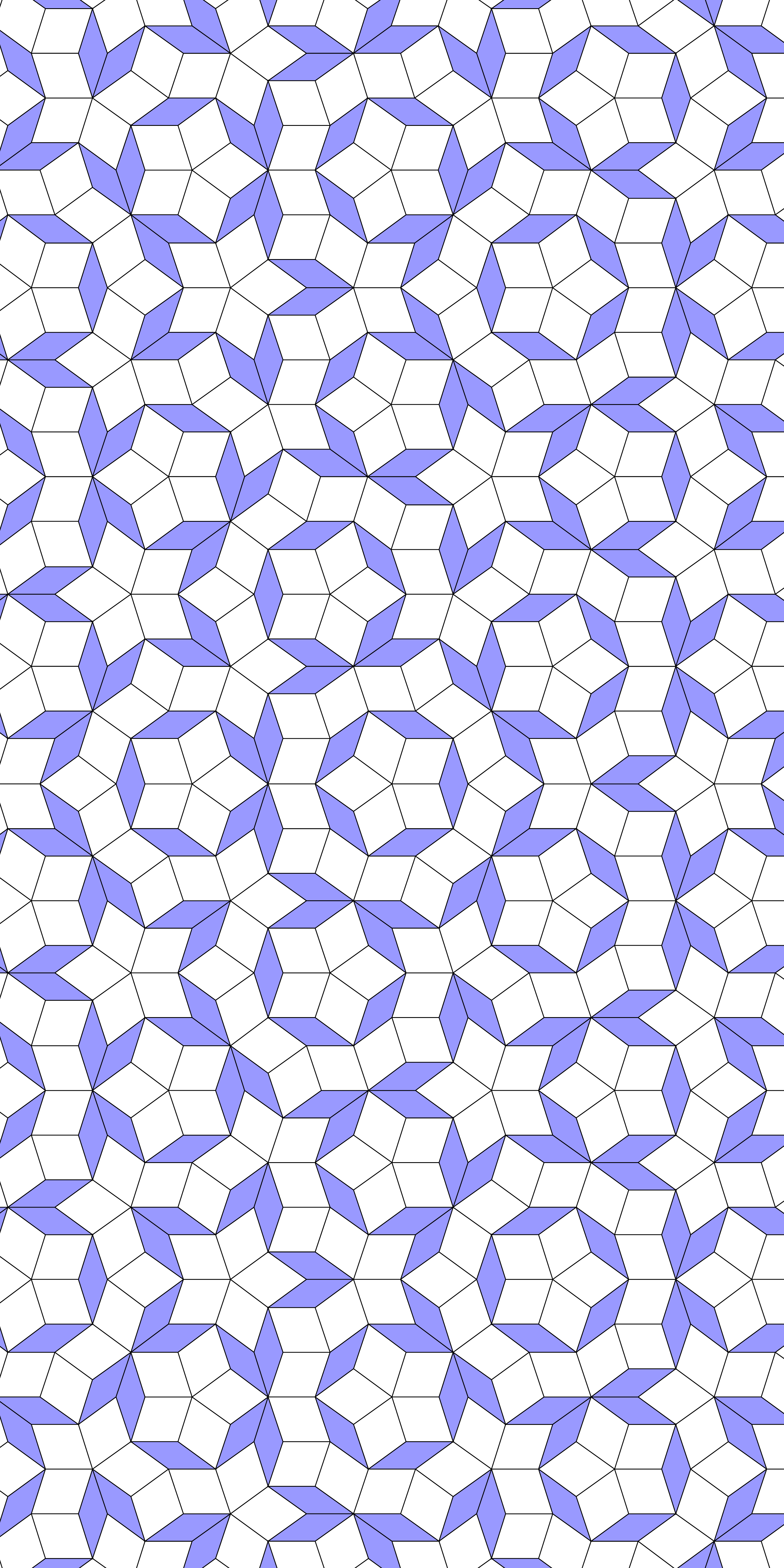}
\caption{From left to right: $6$-fold, $8$-fold and $10$-fold tilings.}
\label{fig:n_fold_tilings}
\end{figure}

\paragraph{Weak local rule.}
Given a tiling $\mathcal{T}$ and a closed ball of radius $r\geq 0$, the tiles of $\mathcal{T}$ that intersect this ball form a pattern called a {\em $r$-map} of $\mathcal{T}$.
The finite set of all the $r$-maps of $\mathcal{T}$ (considered up to a translation) defines the {\em $r$-atlas} of $\mathcal{T}$, denoted by $\mathcal{T}(r)$.
A thickness $1$ planar rhombus tiling $\mathcal{P}$ is then said to admit {\em weak local rules} if there are $r\geq 0$ and $t\geq 1$ such that any rhombus tiling $\mathcal{T}$ with $\mathcal{T}(r)\subset\mathcal{P}(r)$ is planar with the same slope as $\mathcal{P}$ and thickness at most $t$.
In other words, a planar tiling admits weak local rules if its slope is characterized by its patterns of a finite given size.
Fig.~\ref{fig:allowed_patterns} illustrates this.

\begin{figure}[hbtp]
\includegraphics[width=0.25\textwidth]{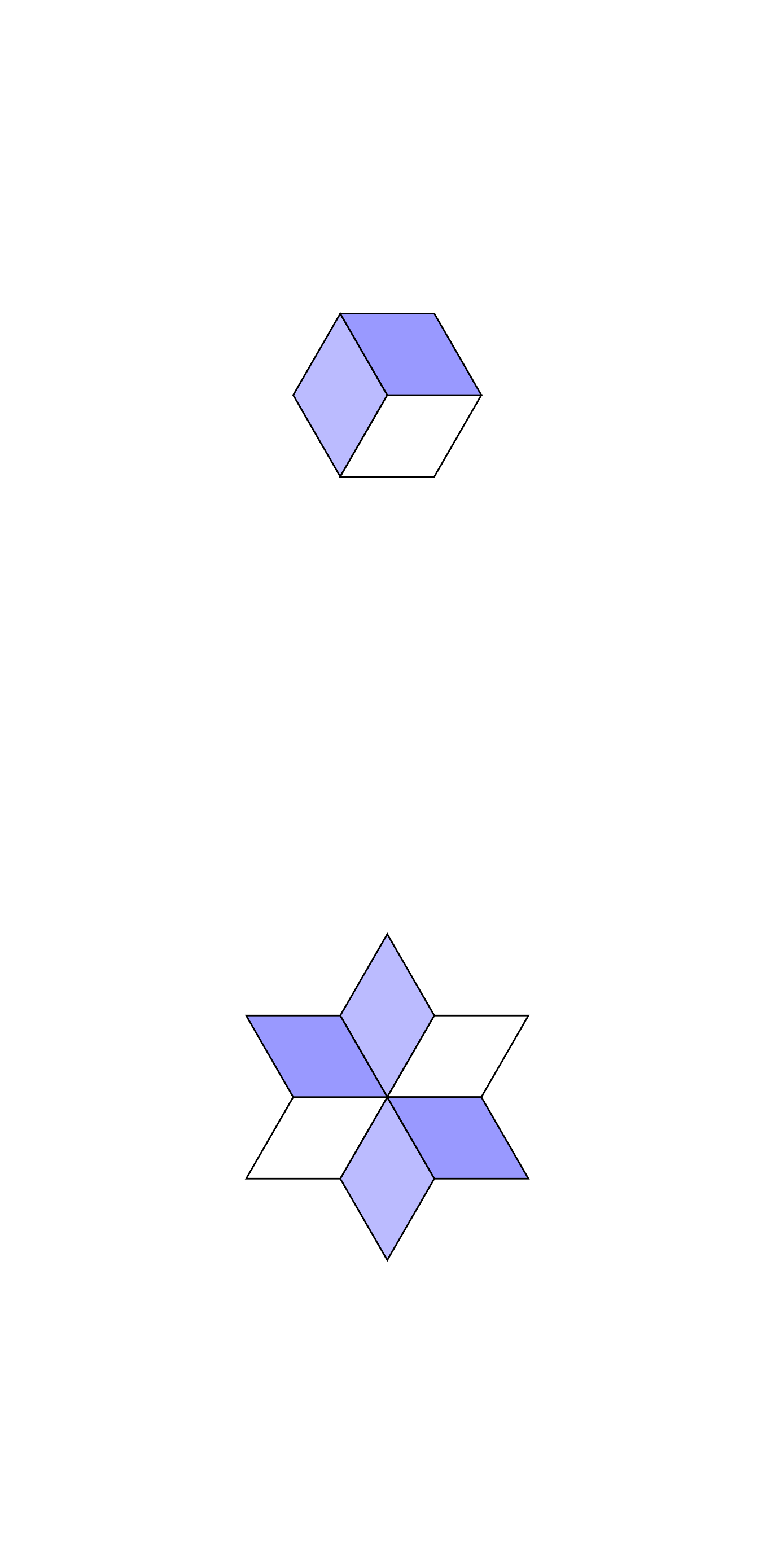}\hfill
\includegraphics[width=0.25\textwidth]{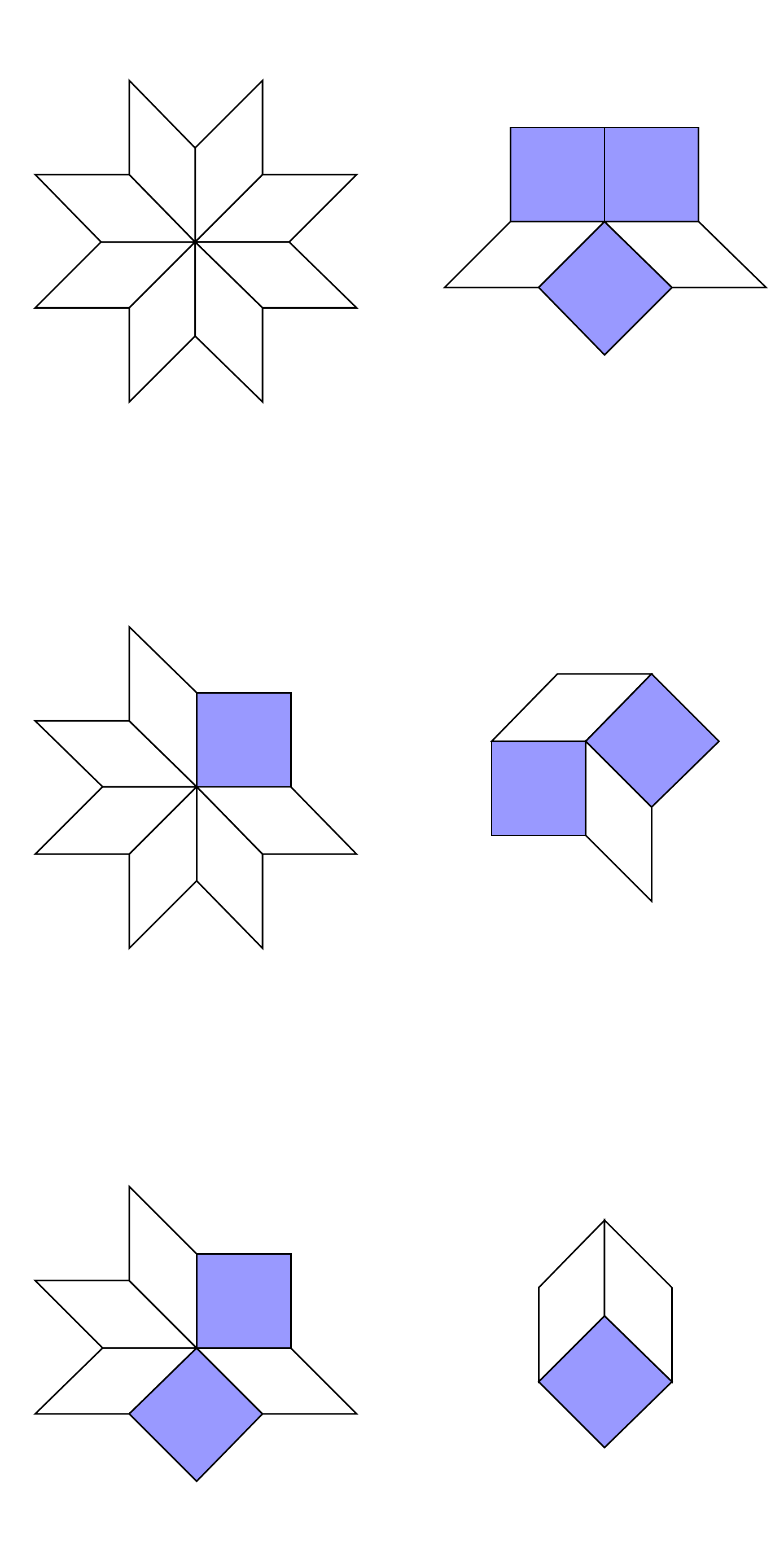}\hfill
\includegraphics[width=0.25\textwidth]{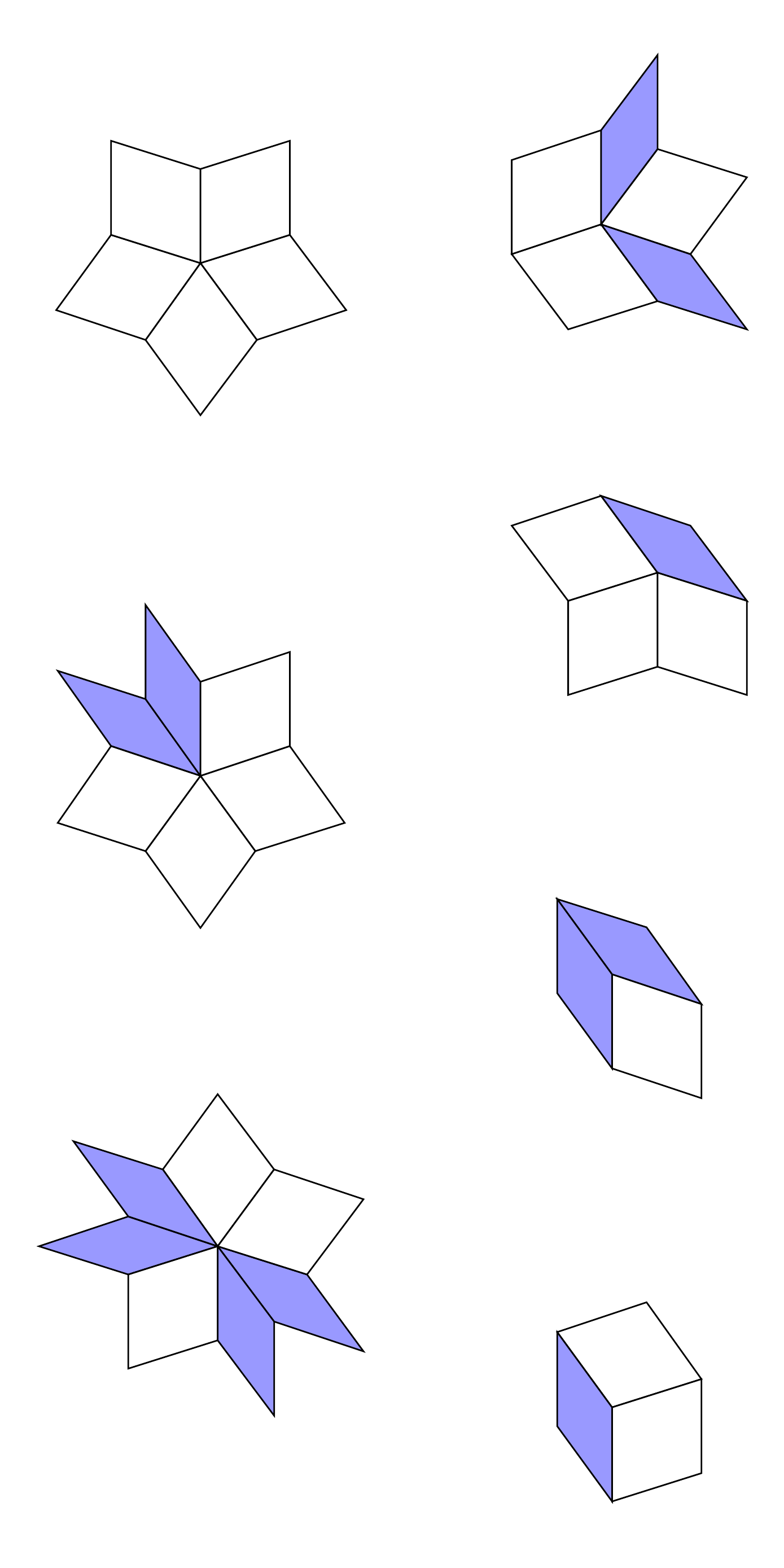}
\caption{
From left to right, the $0$-atlas (also called {\em vertex atlas}) of the $6$-fold, $8$-fold and $10$-fold tilings (up to a rotation).
Compare with Fig.~\ref{fig:n_fold_tilings}.
It is easy to see that the $6$-fold tilings are charaterized by their $0$-atlas.
It is known (see, {\em e.g.}, \cite{senechal}, Th.~6.1 p.~177) that the same holds for the $10$-fold tilings.
On the contrary, Burkov proved in \cite{burkov} that this does not hold for the $8$-fold tilings.
}
\label{fig:allowed_patterns}
\end{figure}

\paragraph{Grassmann coordinate.}
Let $\mathbb{G}(2,n)$ denote the set of the two-dimensional planes in $\mathbb{R}^n$.
If $E\in\mathbb{G}(2,n)$ is generated by $(u_1,\ldots,u_n)$ and $(v_1,\ldots,v_n)$, then its Grassmann coordinates are the $\binom{n}{2}$ real numbers
$$
G_{ij}:=u_iv_j-u_jv_i,
$$
for $i<j$.
%We write $E=(G_{ij})_{i<j}$, with the Grassmann coordinates being ordered by lexicographic order on their indices.
In the case of the $n$-fold tilings:
$$
G_{ij}=\sin\left(\frac{2(j-i)\pi}{n}\right).
$$
The Grassmann coordinates are defined up to a common multiplicative factor and turn out to not depend on the choice of the generating vectors.
Moreover, a non-zero $\binom{n}{2}$-tuple of reals are the Grassmann coordinates of some plane if and only if they satisfy, for any $i<j<k<l$, the so-called {\em Plücker relation}:
$$
G_{ij}G_{kl}=G_{ik}G_{jl}-G_{il}G_{jk}.
$$
By extension, we call Grassmann coordinates of a planar rhombus tiling the Grassmann coordinates of its slope.
They can actually be ``read'' on the tiles: one can indeed show that the frequencies of the $T_{ij}$'s in a planar rhombus tiling are given by the absolute values of the $G_{ij}$'s (up to normalization).
The sign of $G_{ij}$ is equal to the sign of $\det(\vec{v}_i,\vec{v}_j)$, where $\vec{v}_i$ and $\vec{v}_j$ are the vectors of the Euclidean plane which define the tile $T_{ij}$: it is thus independant of the slope.

\paragraph{Non-degeneration.}
A rhombus tiling is said to be {\em nondegenerate} if it contains at least one tile $T_{ij}$ for any $i<j$.
In particular, a planar tiling is nondegenerate if and only if its slope has only non-zero Grassmann coordinates.
The $n$-fold tilings are nondegenerate.
In what follows, we shall implicitly consider only nondegenerate tilings.

\paragraph{Subperiod.}
An {\em $ijk$-subperiod} of a plane $E\in\mathbb{G}(2,n)$ is a non-zero integer vector $(p,q,r)\in\mathbb{Z}^3$ which is a prime period of the orthogonal projection of $E$ onto the three basis vectors $\vec{e}_i$, $\vec{e}_j$ and $\vec{e}_k$.
In terms of Grassmann coordinates, this corresponds to the linear relation
$$
pG_{jk}-qG_{ik}+rG_{ij}=0.
$$
By extension, we call subperiod of a planar rhombus tiling any subperiod of its slope.
It corresponds to a periodic direction in the orthogonal projection on three basis vectors of the tiling lift.
Fig.~\ref{fig:subperiods} illustrates this.
The motivation to introduce subperiods in \cite{subperiods} was to find weak local rules for planar tilings.
We shall use them here, on the contrary, to show that some tilings have no weak local rules.

\begin{figure}[hbtp]
\includegraphics[width=0.23\textwidth]{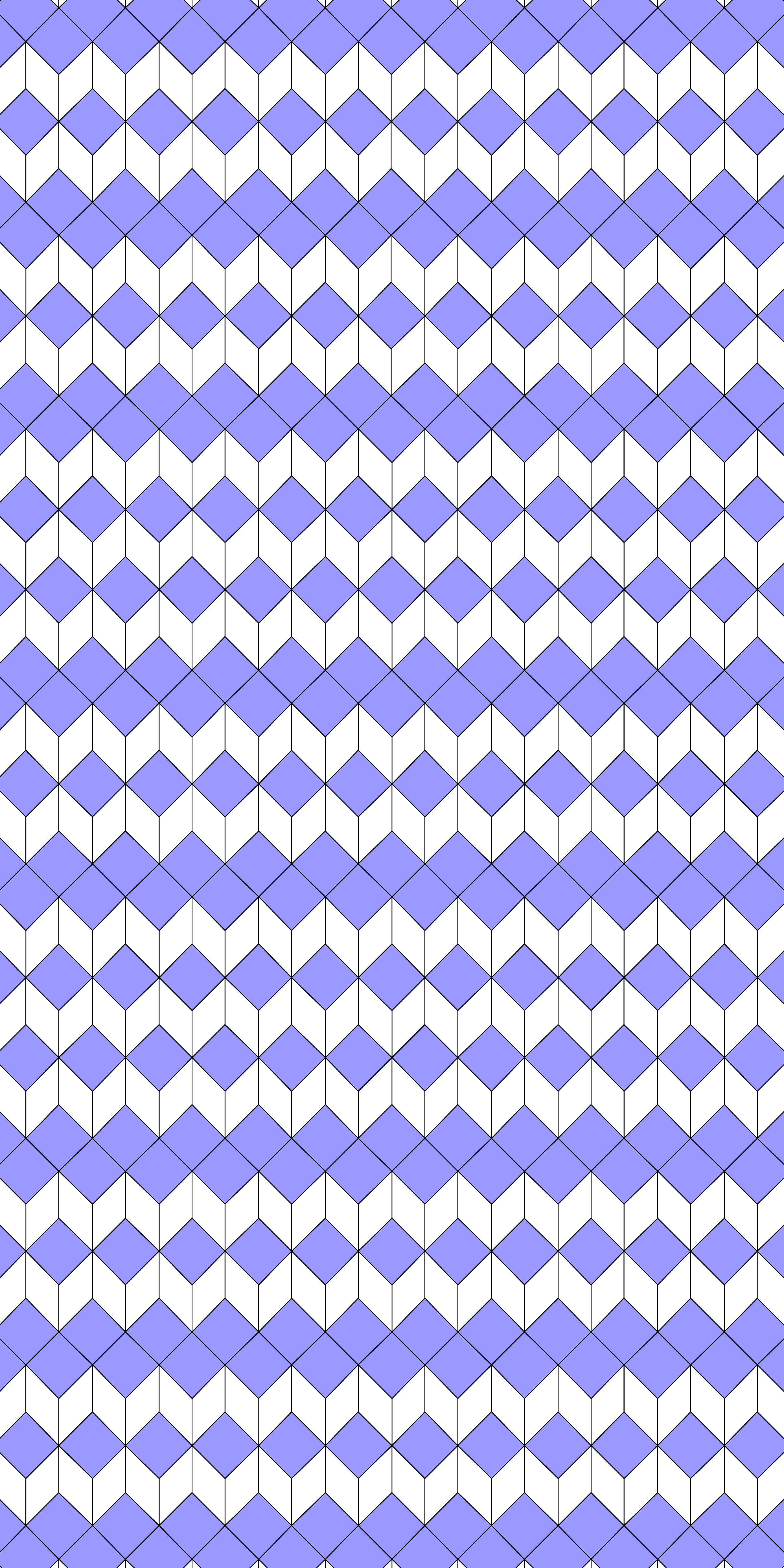}\hfill
\includegraphics[width=0.23\textwidth]{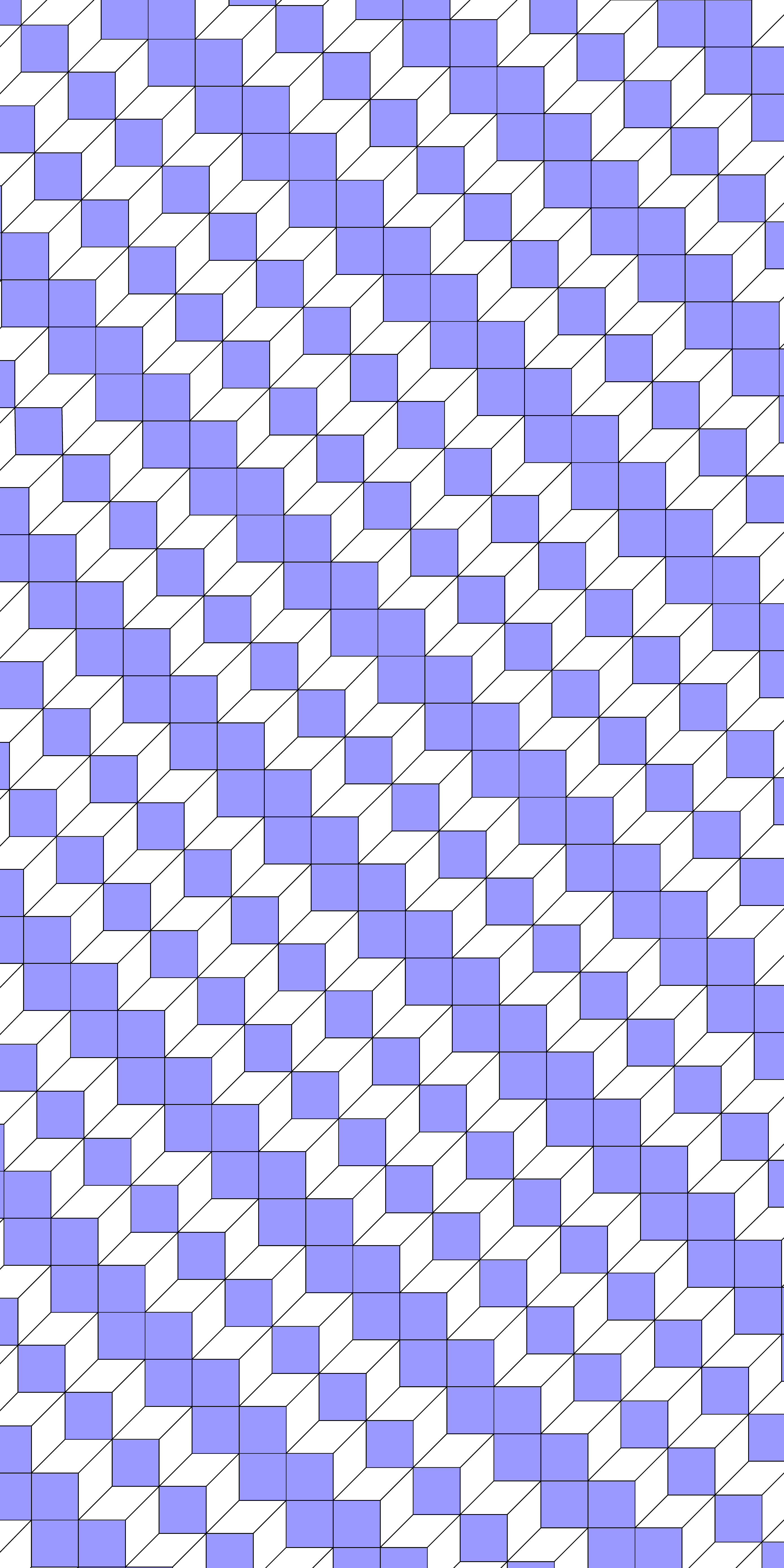}\hfill
\includegraphics[width=0.23\textwidth]{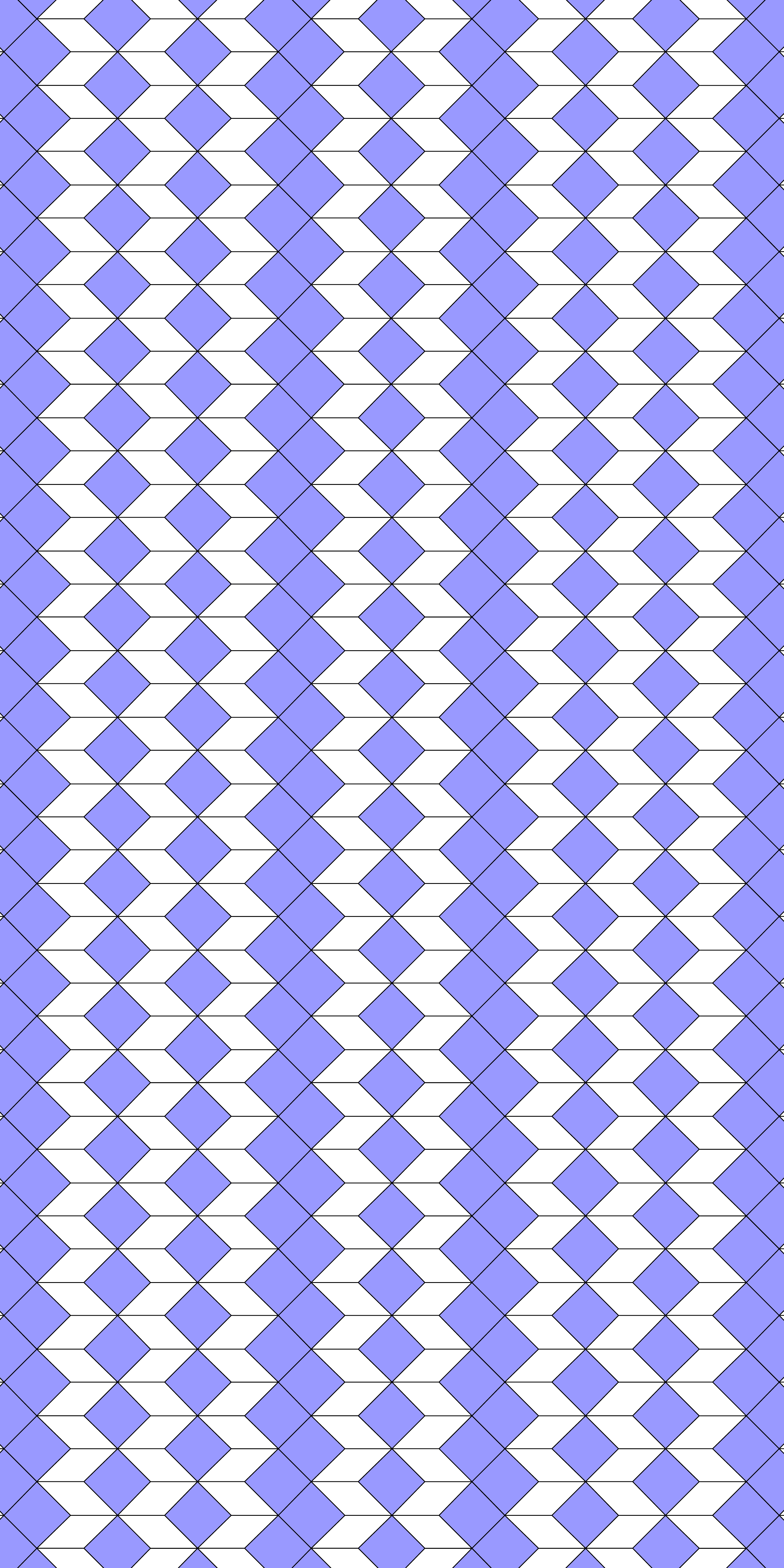}\hfill
\includegraphics[width=0.23\textwidth]{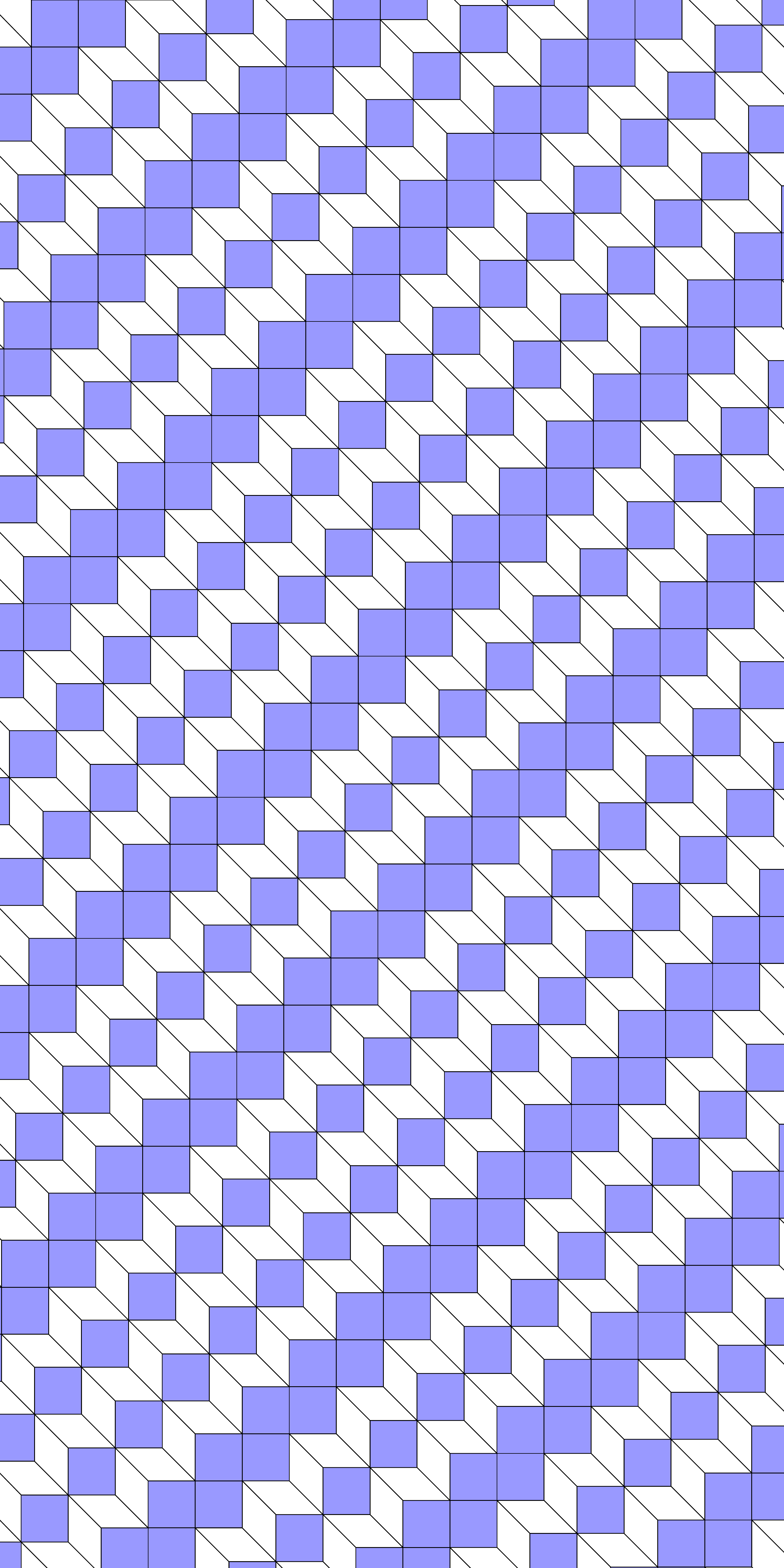}
\caption{The four shadows of an $8$-fold tiling. Each one is periodic.}
\label{fig:subperiods}
\end{figure}

%%%%%%%%%%%%%%%%%%%%%%%%%%
\section{Subperiods of $4p$-fold tilings}
\label{sec:subperiods}

The following proposition is proven in \cite{subperiods}.
We recall it with its proof in order to make the subsequent result more precise.

\begin{proposition}
\label{prop:4p_fold_family1}
The slope of the $4p$-fold tilings belongs to a one-parameter fa\-mi\-ly of slopes which have at least the subperiods of the $4p$-fold tilings.
\end{proposition}

\begin{proof}
The following relations correspond to subperiods of the $4p$-fold tilings:
\begin{eqnarray*}
&&G_{12}=G_{23}=\ldots=G_{2p,2p+1},\\
&&G_{13}=G_{35}=\ldots=G_{2p-1,2p+1},\\
&&G_{24}=G_{46}=\ldots=G_{2p,2p+2},
\end{eqnarray*}
with the convention $G_{i,j+2p}=-G_{i,j}$ and $G_{ji}=-G_{ij}$.
We normalize to $G_{12}=1$ and introduce $X:=\frac{1}{2}G_{13}$, $Y:=\frac{1}{2}G_{24}$ and $U_i:=G_{1,i+2}$.
The Plücker relation
$$
G_{1,i}G_{i+1,i+2}=G_{1,i+1}G_{i,i+2}-G_{1,i+2}G_{i,i+1}
$$
yields the recurrence relation
$$
U_0=1,\quad
U_1=2X,\quad
U_{2i}=2Y U_{2i-1}-U_{2i-2},\quad
U_{2i+1}=2X U_{2i}-U_{2i-1}.
$$
This reminds us of the recurrence defining Chebyshev polynomials of the se\-cond kind.
Precisely, $U_i$ is obtained from the $i$-th Chebyshev polynomial of the second kind by replacing $X^{2k+1}$ by $X^{k+1}Y^k$ and $X^{2k}$ by $X^kY^k$.
In particular, $U_{2p-2}$ is a polynomial of $XY$, and since $U_{2p-2}=G_{1,2p}=G_{2p,2p+1}=1$, there are only finitely many possible values for $XY$.
One shows by induction using Plücker relations that $X$ and $Y$ determine all the other Grassmann coordinates (see \cite{subperiods}, Lem.~4).
The $4p$-fold tilings correspond to $G_{13}=G_{24}$, that is, $XY=\cos^2(\frac{\pi}{2p})$.
This value of $XY$ yields the wanted one-parameter family.
\end{proof}

\begin{figure}[hbtp]
\includegraphics[width=0.3\textwidth]{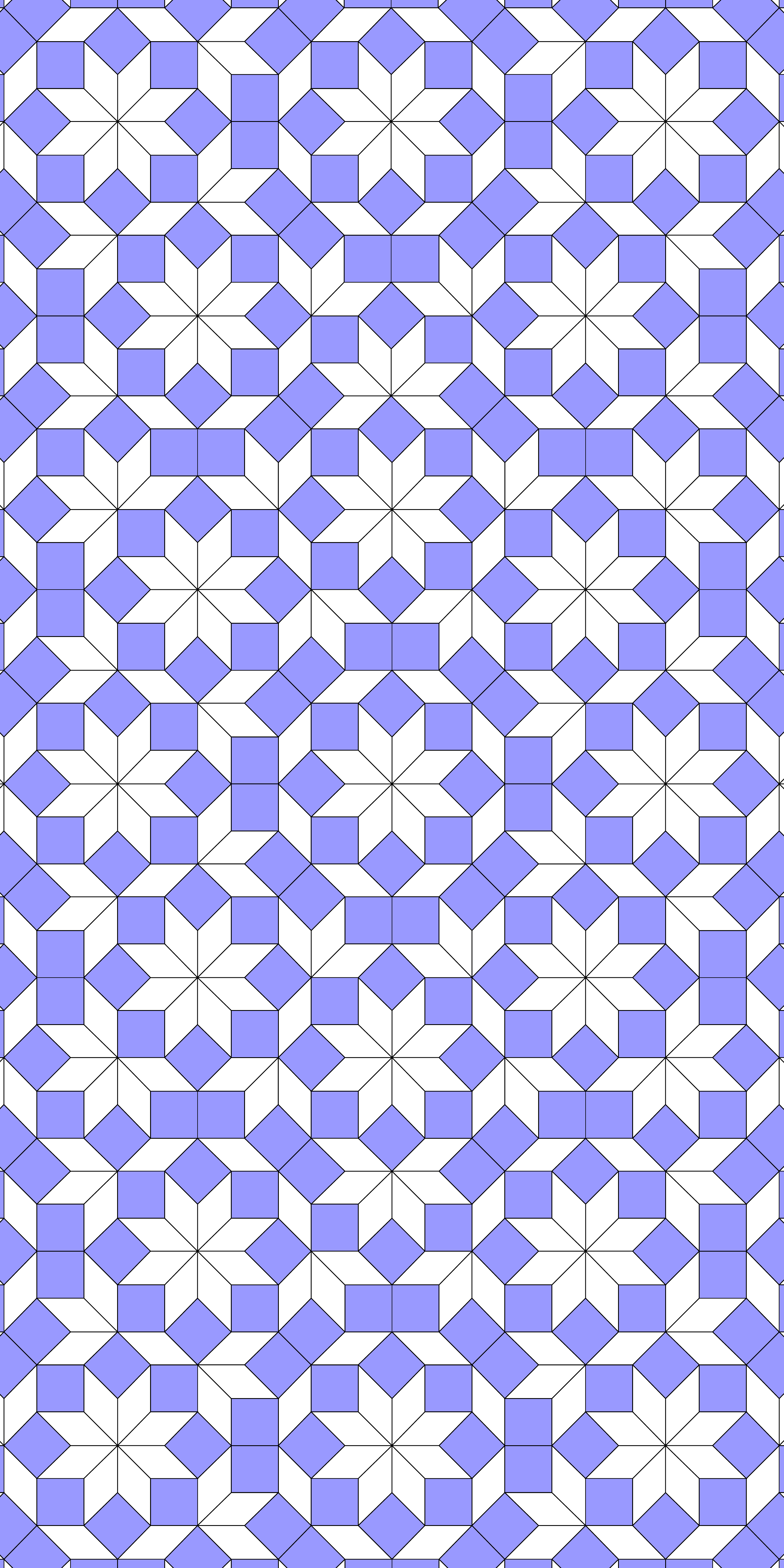}\hfill
\includegraphics[width=0.3\textwidth]{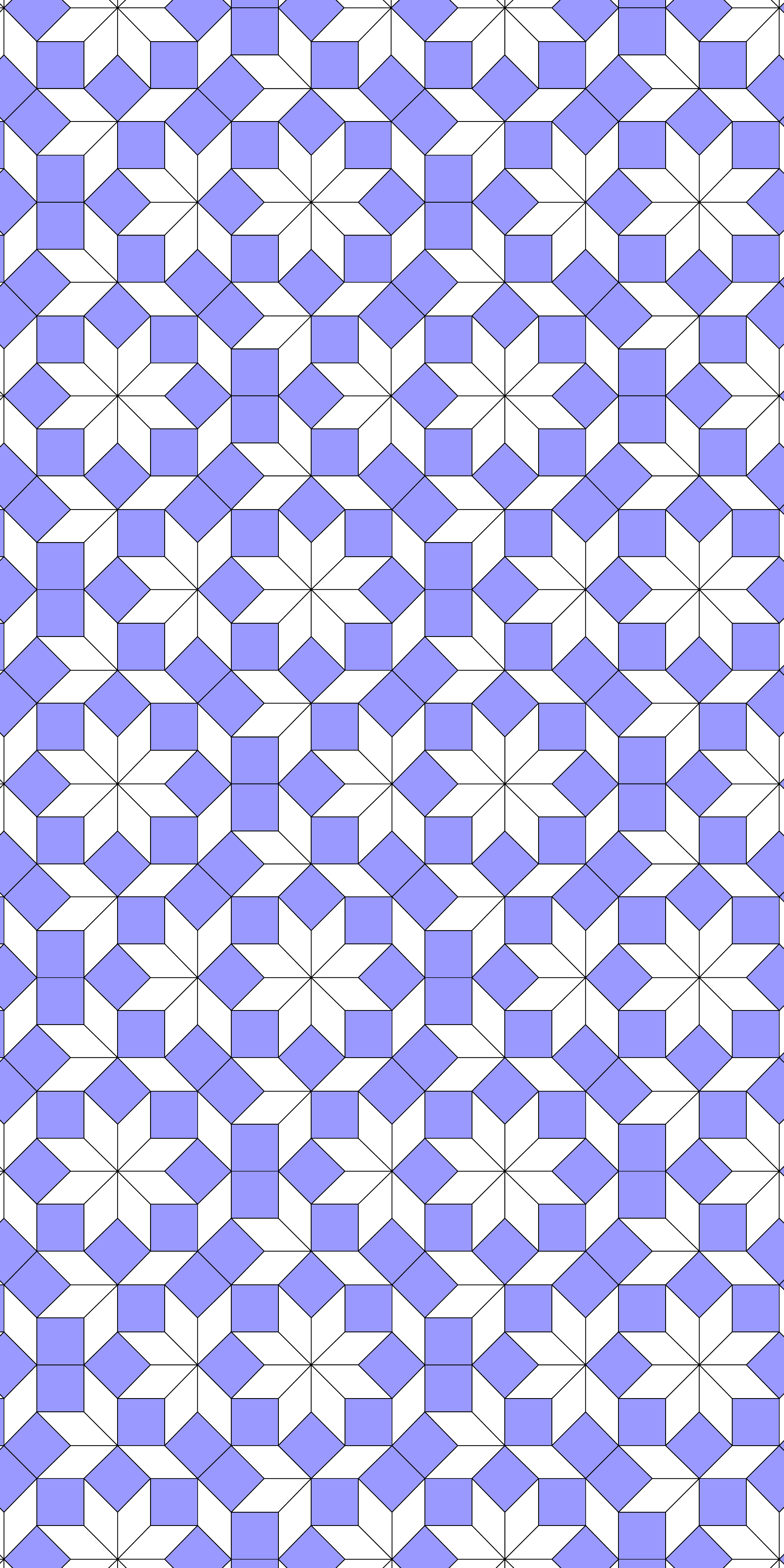}\hfill
\includegraphics[width=0.3\textwidth]{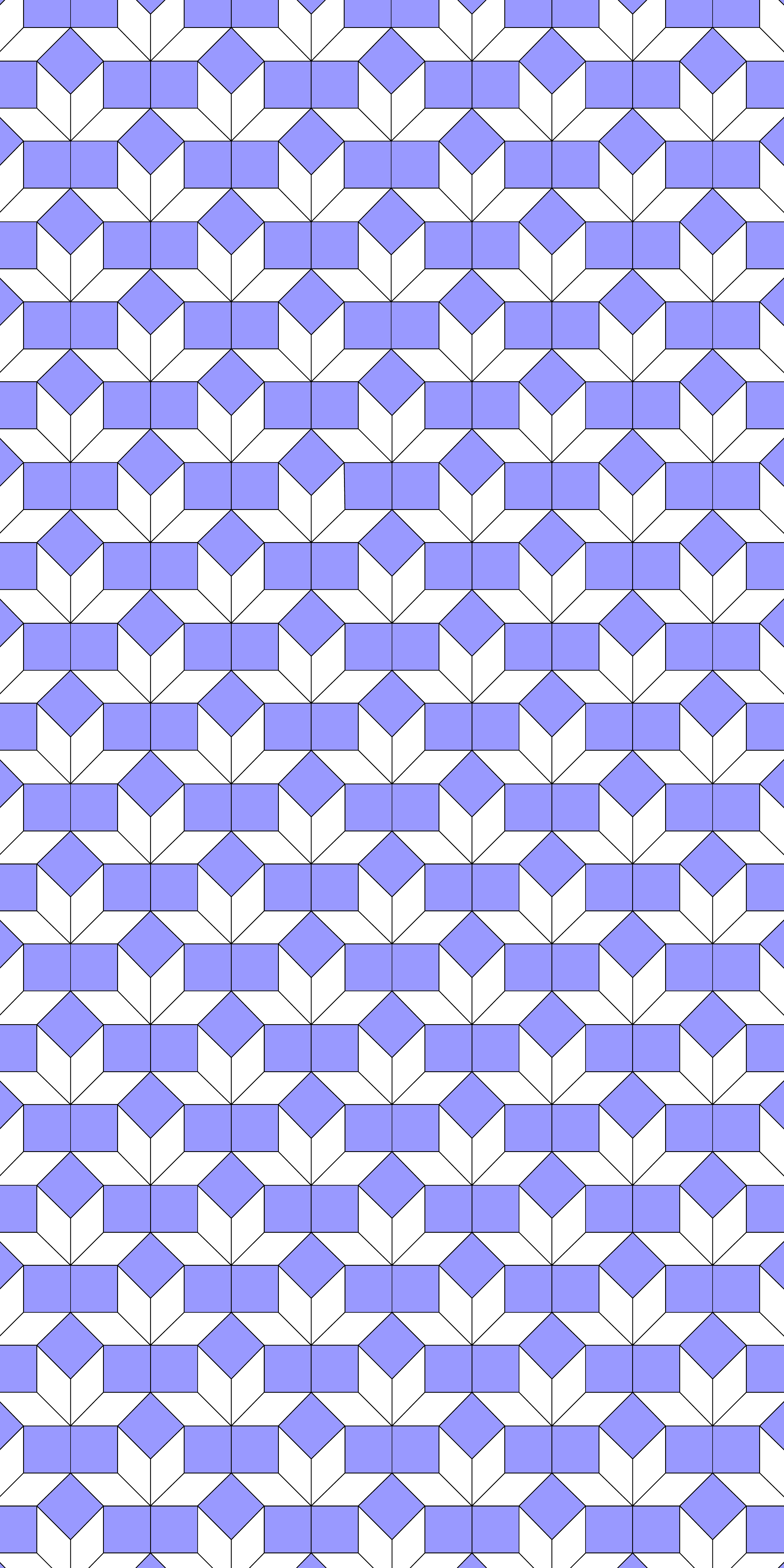}
\caption{
Some planar tilings with the same subperiods as the $8$-fold tilings.
The left one is a $8$-fold tiling and has slope $E_{t_2}=E_{\sqrt{2}}$.
The middle and the right ones respectively have slope $E_{\frac{3}{2}}$ and $E_1$.
They are not $8$-fold tilings, although the middle one has the same $0$-atlas as the $8$-fold tilings (compare with Fig.~\ref{fig:allowed_patterns}).
}
\label{fig:8_fold_tilings}
\end{figure}

Consider the one-parameter family of slopes found in Prop.~\ref{prop:4p_fold_family1}.
We denote by $E_t$ the slope with $G_{12}=1$ and $G_{13}=t$.
The $4p$-fold tilings thus correspond to $t=t_p:=2\cos(\frac{\pi}{2p})$.
Let us give a basis of $E_t$ that shall be useful.

\begin{proposition}
\label{prop:4p_fold_family2}
There are two vectors with entries in $\mathbb{Q}(t_p^2)$ such that, for any $t$, multiplying by $t$ their entries with an odd index\footnote{The first index is one.} yields a basis of $E_t$.
\end{proposition}

\begin{proof}
We keep the normalization $G_{12}=1$ and the parametrization $G_{13}=t$.
Let us show by induction on $j-i$ the following claim:
\begin{itemize}
\item if $j-i$ is odd, then $G_{ij}\in\mathbb{Q}(t_p^2)$;
\item if $j-i$ is even and $i$ is even, then $G_{ij}\in\mathbb{Q}(t_p^2)/t$;
\item if $j-i$ is even and $i$ is odd, then $G_{ij}\in\mathbb{Q}(t_p^2)t$.
\end{itemize}
This holds for $j-i\leq 2$ since $G_{i,i+1}=G_{12}=1$, $G_{2i+1,2i+3}=G_{13}=t$ and $G_{2i,2i+2}=G_{24}=t_p^2/t$ (because $G_{13}G_{24}=4XY=t_p^2$ in the proof of Prop.~\ref{prop:4p_fold_family1}).
Assume that this claim holds for $j-i<\delta$ and consider $i$ and $j$ such that $j-i=\delta$.
We rely on the Plücker relation
$$
G_{i,j-1}G_{i+1,j}-G_{ij}G_{i+1,j-1}=G_{i,i+1}G_{j-1,j}=1.
$$
\begin{itemize}
\item if $j-i$ is even and $i$ is even:
$$
\underbrace{G_{i,j-1}}_{\mathbb{Q}(t_p^2)}\underbrace{G_{i+1,j}}_{\mathbb{Q}(t_p^2)}-G_{ij}\underbrace{G_{i+1,j-1}}_{\mathbb{Q}(t_p^2)t}=1.
$$
\item if $j-i$ is even and $i$ is odd:
$$
\underbrace{G_{i,j-1}}_{\mathbb{Q}(t_p^2)}\underbrace{G_{i+1,j}}_{\mathbb{Q}(t_p^2)}-G_{ij}\underbrace{G_{i+1,j-1}}_{\mathbb{Q}(t_p^2)/t}=1.
$$
\item if $j-i$ is odd, with $\varepsilon=1$ if $i$ is odd or $\varepsilon=-1$ otherwise:
$$
\underbrace{G_{i,j-1}}_{\mathbb{Q}(t_p^2)t^\varepsilon}\underbrace{G_{i+1,j}}_{\mathbb{Q}(t_p^2)/t^\varepsilon}-G_{ij}\underbrace{G_{i+1,j-1}}_{\mathbb{Q}(t_p^2)}=1.
$$
\end{itemize}
In any case, the claim holds for $G_{ij}$, hence by induction for any $i<j$.\\
Now, consider the two following vectors
$$
(-G_{12},0,G_{23},G_{24},\ldots,G_{2,2p})
\qquad\textrm{and}\qquad
(0,G_{12},G_{13},\ldots,G_{1,2p}).
$$
One checks that they form a basis of $E_t$.
We get the two wanted vectors by multiplying by $t$ the even entries of the first vector and by dividing by $t$ the odd entries of the second vector.
\end{proof}

\noindent Let us illustrate this for the first values of $p$:
\begin{itemize}
\item
For $p=2$, consider the vectors
$$
\vec{u}_2:=(-1,0,1,2)
\qquad\textrm{and}\qquad
\vec{v}_2:=(0,1,1,1).
$$
Both have entries in $\mathbb{Q}(t_2^2)=\mathbb{Q}$.
Multiplying by $t$ their odd entries yields the following basis of $E_t$
$$
\vec{u}_2(t):=(-t,0,t,2)
\qquad\textrm{and}\qquad
\vec{v}_2(t):=(0,1,t,1).
$$
The $8$-fold tilings have slope $E_{t_2}=E_{\sqrt{2}}$.
\item
For $p=3$, consider the vectors
$$
\vec{u}_3:=(-1,0,1,3,2,3)
\qquad\textrm{and}\qquad
\vec{v}_3:=(0,1,1,2,1,1).
$$
Both have entries in $\mathbb{Q}(t_3^2)=\mathbb{Q}$.
Multiplying by $t$ their odd entries yields a basis of $E_t$.
The $12$-fold tilings have slope $E_{t_3}=E_{\sqrt{3}}$.
\item
For $p=4$, consider the two vectors
$$
\vec{u}_4:=(-1,0,1,2+\sqrt{2},1+\sqrt{2},1+\sqrt{2},1+\sqrt{2},2+\sqrt{2})
$$
$$
\textrm{and}\qquad \vec{v}_4:=(0,1,1,1+\sqrt{2},\sqrt{2},1+\sqrt{2},1,1).
$$
Both have entries in $\mathbb{Q}(t_4^2)=\mathbb{Q}(\sqrt{2})$.
Multiplying by $t$ their odd entries yields a basis of $E_t$.
The $16$-fold tilings have slope $E_{t_4}=E_{\sqrt{2+\sqrt{2}}}$.
\end{itemize}
%%%%%%%%%%%%%%%%%%%%%%%%%%
\section{In the window}
\label{sec:window}

Let us first briefly recall how the shape of the patterns of a planar tiling is governed by the way the vertices of its lift project onto the space orthogonal to its slope (also called {\em internal space}).
We follow \cite{julien}, where more details as well as proofs of the results here recalled can be found.\\

Let $E\in\mathbb{G}(2,n)$ be a two-dimensional plane in $\mathbb{R}^n$.
The orthogonal projection of the unit hypercube $[0,1]^n$ onto $E^\bot$ is called the {\em window}.
The vertices of the lifts of planar tilings of slope $E$ and thickness $1$ are precisely the points of $\mathbb{Z}^n$ whose orthogonal projection onto $E^\bot$ lies in the window.
Then, let $S_k$ be the set of the unit faces of $\mathbb{Z}^n$ of dimension $n-3$ lying in $[0,k]^n$.
The orthogonal projection of $S_k$ onto $E^\bot$ yields a union of codimension $1$ faces which divide the window in convex polytopes.
There is a bijective correspondance between these polytopes and the patterns of the planar tilings of slope $E$ and thickness $1$.
% (one can moreover show that the volume of each polytope is proportional to the frequency of the corresponding pattern).
% pas tjrs car les points ne sont pas forcément uniformément répartis
Namely, given a vertex $x$ of the lift of such a tiling, the restriction of this lift to $x+[-k,k]^n$ depends only on the convex polytope the orthogonal projection of $x$ onto $E^\bot$ falls in.
Fig.~\ref{fig:window_patterns} illustrates this in the $n=4$ case with an $8$-fold tiling.\\

\begin{figure}[hbtp]
\centering
\includegraphics[width=0.9\textwidth]{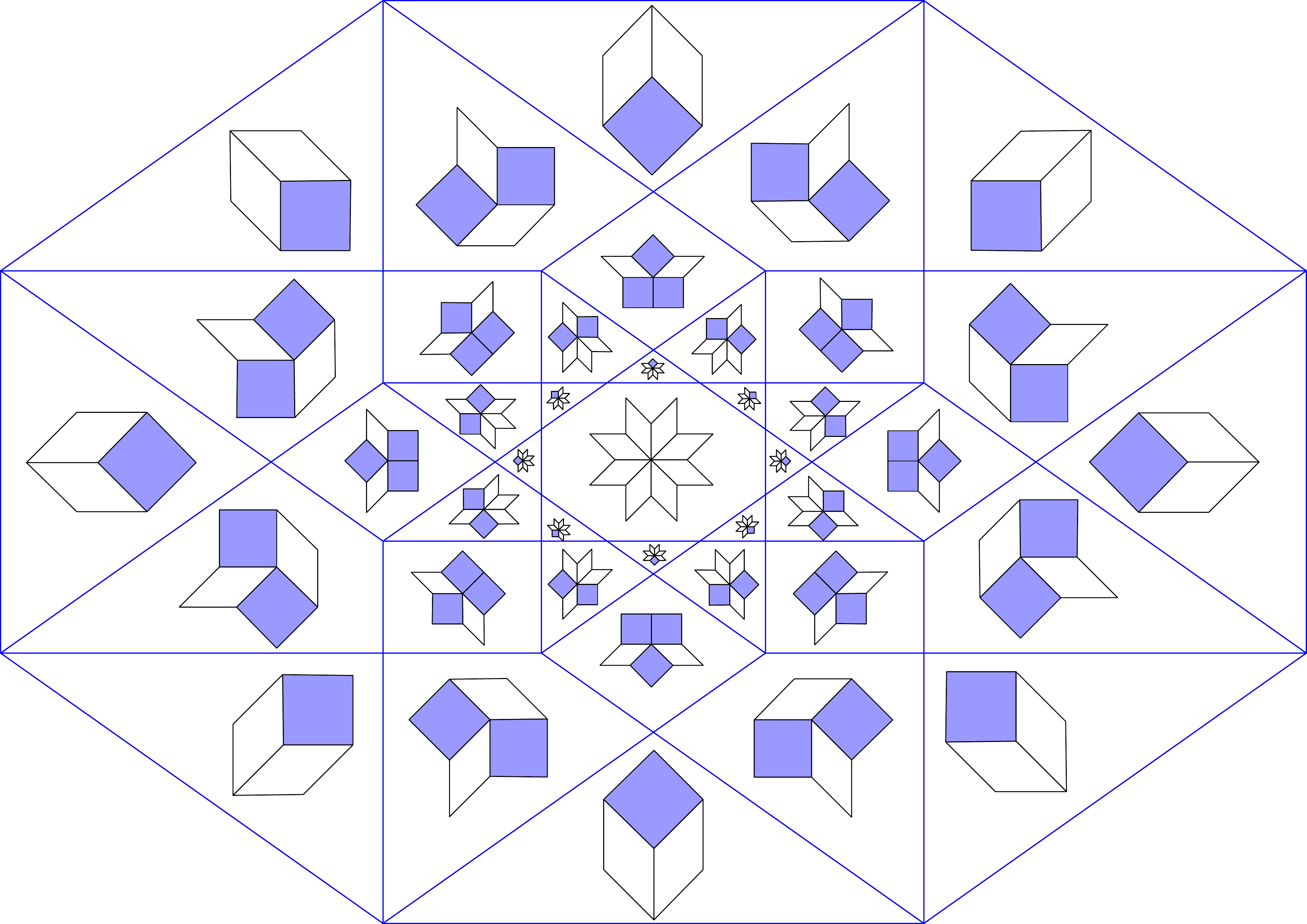}
\caption{
The division of the window by $S_1$ for an $8$-fold tiling of slope $E_{\sqrt{2}}$.
Whenever a vertex projects orthogonally onto $E_{\sqrt{2}}^\bot$ into one of these regions, its orthogonal projection onto $E_{\sqrt{2}}$ is the center of the $0$-map drawn in this region.
}
\label{fig:window_patterns}
\end{figure}

We are interested in how the patterns are modified when the slope varies, that is, what happens in the window.
The notion of {\em coincidence} shall be useful:

\begin{definition}\label{def:coincidence}
A {\em coincidence} of $E\in\mathbb{G}(2,n)$ is a set of $n-1$ unit faces of $\mathbb{Z}^n$ of dim. $n-3$ whose orthogonal projections onto $E^\bot$ have a non-empty intersection.
\end{definition}

\noindent The following proposition is illustrated in the $n=4$ case by Figure~\ref{fig:window_move}:

\begin{proposition}\label{prop:coincidence}
Let $E$ be a plane in $\mathbb{G}(2,n)$.
Assume that $E$ belongs to a curve of $\mathbb{G}(2,n)$ such that any coincidence of $E$ is also a coincidence of the points of this curve which are close enough to $E$.
Then $E$ does not admit weak local rules.
\end{proposition}

\begin{proof}
Let $(E_t)_t$ be such a curve, with $E=E_0$.
Fix $k>0$.
The hypotheses ensure that for $t$ small enough, all the (finetely many) coincidences of $E$ formed by faces in $S_k$ are coincidences of $E_t$.
Assume that there is a pattern of size $k$ which appears in $E_t$ but not in $E$.
The corresponding connected component in the window of $E_t$ thus shrinks when $t$ decreases until its interior vanishes for $t=0$.
This connected component is a polytope in a $(n-2)$-dimensional space: its faces are projections of faces in $S_k$ and its vertices are intersections of $n-2$ such faces.
These vertices move with $t$ until entering a new face when the interior of the polytope vanishes for $t=0$.
This yields $n-1$ intersecting face which are the projections of unit faces of $\mathbb{Z}^n$ of dimension $n-3$, that is, a new coincidence for $t=0$.
Since the hypotheses prevent that, this means that any pattern of size $k$ of $E$ also appears in $E_t$.
Thus, the planar tilings of slope $E$ and $E_t$ cannot be distinguished by such patterns.
Since this holds for any $k$, this ensures that $E$ does not admit weak local rules.
\end{proof}

\begin{figure}[hbtp]
\includegraphics[width=0.98\textwidth]{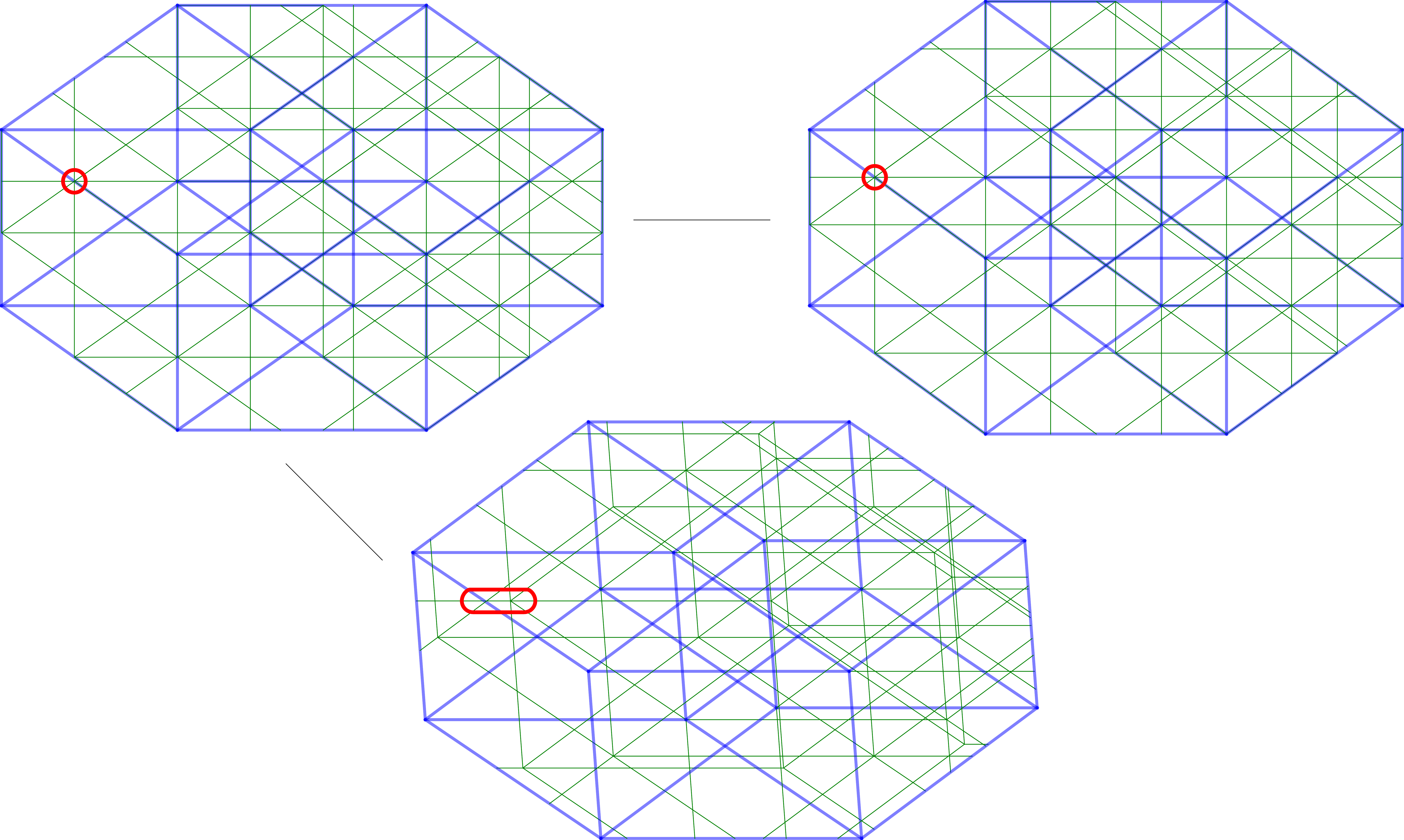}
\caption{
Top-left, the division by $S_2$ of the window of a $8$-fold tiling, with a circled coincidence.
Top-right, this coincidence is preserved by slightly moving the slope {\em along} the curve of the slopes having the same subperiods.
Bottom, the coincidence breaks by slightly moving the slope {\em transversally} to this curve.
%This is not the case on the bottom window, obtained by slightly moving the slope {\em transversally} to this curve.
%Top-left, the division by $S_2$ of the window of a $8$-fold tiling of slope $E_{\sqrt{2}}$, with a circled coincidence.
%Top-right, this coincidence is preserved by slightly moving the slope {\em along} the curve $(E_t)_t$ of the slopes with the same subperiods as $E_{\sqrt{2}}$.
%This is not the case on the bottom window, obtained by slightly moving the slope {\em transversally} to the curve $(E_t)_t$.
}
\label{fig:window_move}
\end{figure}

%%%%%%%%%%%%%%%%%%%%%%%%%%
\section{Coincidences of $4p$-fold tilings}
\label{sec:coincidences}

We here prove Theorem~\ref{th:main} by showing (Lemma~\ref{lem:4p_family_coincidences}) that the one-parameter family of planar tilings with the same subperiods as the $4p$-fold tilings (Proposition~\ref{prop:4p_fold_family1}) forms a curve of $\mathbb{G}(2,2p)$ which fulfills the hypotheses of Proposition~\ref{prop:coincidence}.
We first need an algebraic lemma which shall be used in the proof of Lemma~\ref{lem:4p_family_coincidences}

\begin{lemma}
\label{lem:incommensurables}
For $p\geq 2$, the parameter $t_p=2\cos(\frac{\pi}{2p})$ does not belong to $\mathbb{Q}(t_p^2)$.
\end{lemma}

\begin{proof}
Since $t_p^2=4\cos^2(\frac{\pi}{2p})=2+2\cos(\frac{\pi}{p})$, let us show $\cos(\frac{\pi}{2p})\notin\mathbb{Q}(\cos(\frac{\pi}{p}))$.
Recall that $\cos(\frac{\pi}{p})$ is an algebraic number of degree $\frac{\varphi(2p)}{2}$, where $\varphi$ is the Euler's totient function.
The algebraic degree of $\cos(\frac{\pi}{2p})$ is thus $\frac{\varphi(4p)}{2}=\varphi(2p)$.
It does not divide $\frac{\varphi(2p)}{2}$.
The result follows since the algebraic degree of any element in a field extension divides the algebraic degree of this extension.
\end{proof}

\begin{lemma}\label{lem:4p_family_coincidences}
A coincidence of $E_{t_p}$ is a coincidence of $E_t$ for $t$ close enough to $t_p$.
\end{lemma}

\begin{proof}
Consider a coincidence of $E_{t_p}$, that is, a set $F_1,\ldots,F_{2p-1}$ of $(2p-3)$-dimensional unit faces of $\mathbb{Z}^{2p}$ whose orthogonal projections onto $E_{t_p}^\bot$ have a non-empty intersection.
Each face $F_i$ thus contains a point $X_i$ such that the difference of any two such points is in $E_{t_p}$.
Let $\vec{u}(t)$ and $\vec{v}(t)$ denote the basis of $E_t$ obtained by multiplying by $t$ the odd entries of the two vectors of Prop.~\ref{prop:4p_fold_family2}.
For $t=t_p$ and $2\leq j< 2p$, there are thus two real numbers $\lambda_j$ et $\mu_j$ such that
$$
X_1-X_j=\lambda_j\vec{u}(t)+\mu_j\vec{v}(t).
$$
With $x_{i,j}$ denoting the $i$-th entry of $X_j$, this yields $2p(2p-2)$ equations in $t$:
$$
x_{i,1}-x_{i,j}=u_i(t)\lambda_j+v_i(t)\mu_j.
$$
We shall prove that, for $t$ close enough to $t_p$, one can modify the $x_{i,j}$'s so that the above equations are satisfied and each $X_i$ still belongs to the face $F_i$.
These equations fall into exactly three types:
\begin{enumerate}
\item these where both $x_{i,1}$ and $x_{i,j}$ are integers;
\item these where only $x_{i,j}$ is an integer;
\item these where $x_{i,j}$ is not an integer.
\end{enumerate}
We split the proof in three corresponding steps.\\

\noindent {\bf Step 1.}
We show that, for any $j$, there are $a_j$, $b_j$, $c_j$ and $d_j$ in $\mathbb{Q}(t_p^2)$ such that the first type equations are satisfied for $t$ close enough to $t_p$ with
$$
\lambda_j=a_j+\frac{b_j}{t}
\qquad\textrm{and}\qquad
\mu_j=c_j+\frac{d_j}{t}.
$$
Assume that there are two equations of the first type:
\begin{eqnarray*}
x_{i,1}-x_{i,j}&=&u_i(t)\lambda_j+v_i(t)\mu_j,\\
x_{k,1}-x_{k,j}&=&u_k(t)\lambda_j+v_k(t)\mu_j.
\end{eqnarray*}
This is a system in $\lambda_j$ and $\mu_j$ with determinant $u_i(t)v_k(t)-u_k(t)v_i(t)$, which is non-zero for $t=t_p$ and thus also for $t$ close enough to $t_p$ by continuity.
Hence:
\begin{eqnarray*}
\lambda_j&=&\frac{(x_{i,1}-x_{i,j})v_k(t)-(x_{k,1}-x_{k,j})v_i(t)}{u_i(t)v_k(t)-u_k(t)v_i(t)},\\
\mu_j&=&\frac{(x_{i,1}-x_{i,j})u_k(t)-(x_{k,1}-x_{k,j})u_i(t)}{u_k(t)v_i(t)-u_i(t)v_k(t)}.
\end{eqnarray*}
One checks that $\lambda_j$ et $\mu_j$ are in $\mathbb{Q}(t_p^2)$ if $i$ and $k$ are both even, in $\mathbb{Q}(t_p^2)/t$ if they are both odd, and in $\mathbb{Q}(t_p^2)+\mathbb{Q}(t_p^2)/t$ otherwise.
In any case, they can be written as claimed.
This is all the more the case if there is at most one equation of the first type.
Let us now show that any other equation of the first type is automatically satisfied.
Consider such an equation which involves $\lambda_j$ and $\mu_j$:
\begin{eqnarray*}
x_{l,1}-x_{l,j}&=&u_l(t)\lambda_j+v_l(t)\mu_j.
\end{eqnarray*}
Replacing $\lambda_j$ and $\mu_j$ by their expressions yields
$$
(x_{l,1}-x_{l,j})G_{ik}(t)=(x_{k,1}-x_{k,j})G_{il}(t)-(x_{i,1}-x_{i,j})G_{kl}(t),
$$
where $G_{ij}(t)=u_i(t)v_j(t)-u_j(t)v_i(t)$ denotes the Grassmann coordinate of $E_t$.
This is exactly the equation of a subperiod of $E_t$.
It is satisfied for $t=t_p$ and thus for any $t$ because any subperiod of $E_{t_p}$ is also a subperiod of $E_t$.
Last, since none of the $x_{i,j}$'s have been here modified, each $X_i$ is still in $F_i$.\\

\noindent {\bf Step 2.}
We show that, with the above defined $\lambda_j$'s and $\mu_j$'s, there is for any $t$ a vector $X_1$ such that all the equations of the second type are satisfied.
For a given $i$, an equation of the second type characterizes $x_{i,1}$:
$$
x_{i,1}=x_{i,j}+\lambda_j u_i(t)+\mu_j v_i(t).
$$
It thus suffices to check that whenever two such equations characterize the same $x_{i,1}$, they are consistent, that is:
$$
x_{i,j}+\lambda_j u_i(t)+\mu_j v_i(t)=x_{i,k}+\lambda_k u_i(t)+\mu_k v_i(t).
$$
Replacing $\lambda_j$, $\mu_j$, $\lambda_k$ and $\mu_k$ by their expressions yields:
$$
x_{i,j}-x_{i,k}+\left(a_j-a_k+\frac{c_j-c_k}{t}\right)u_i(t)+\left(b_j-b_k+\frac{d_j-d_k}{t}\right)v_i(t)=0.
$$
Whatever the parity of $i$ is, we get an equation of the type $a+bt=0$ with $a$ and $b$ both in $\mathbb{Q}(t_p^2)$.
Lemma~\ref{lem:incommensurables} with $t=t_p$ then yields $a=b=0$.
The equation is thus satisfied for any $t$.
Since $x_{i,1}$ is not an integer and its variation is continuous in $t$, it has still the same floor for $t$ close enough to $t_p$, that is, $X_1$ still belongs to $F_1$.\\

\noindent {\bf Step 3.}
The entry $x_{i,j}$ of an equation of the third type appears only in this equation.
It can thus be freely modified, for any $t$, so that the equation remains satisfied.
Since $x_{i,j}$ is not an integer and its variation is continuous in $t$, it has still the same floor for $t$ close enough to $t_p$, that is, $X_i$ still belongs to $F_i$.
\end{proof}

By combining the above lemma with Proposition~\ref{prop:coincidence}, we finally get a proof of our main result, Theorem~\ref{th:main}.

\end{document}